\documentclass{amsart}

\usepackage{amsmath,amsthm,amssymb}
\usepackage[utf8]{inputenc}

\usepackage{natbib}
\usepackage{hyperref}

\newtheorem{lemma}{Lemma}%[section]
\theoremstyle{definition}
\newtheorem{definition}[lemma]{Definition}
\theoremstyle{remark}

\newtheorem{example}[lemma]{Example}

\newcommand{\Nb}{\mathbb{N}}

\title{Cores, shell indices and the degeneracy of a graph limit}
\author{Johannes Rauh}
\address{MPI for Mathematics in the Sciences\\Inselstraße 23\\04103 Leipzig}
\email{jrauh@mis.mpg.de}

\begin{document}

\begin{abstract}
  The $k$-core of a graph is the maximal subgraph in which every node has degree at least~$k$, the shell index of a
  node is the largest $k$ such that the $k$-core contains the node, and the degeneracy of a graph is the largest shell
  index of any node.  After a suitable normalization, these three concepts generalize to limits of dense graphs (also
  called graphons).  In particular, the degeneracy is continuous with respect to the cut metric.
\end{abstract}

\maketitle{}

\section{Introduction}

The goal of this paper is to study the graph theoretic notions of cores, shell indices and degeneracy from the view
point of graphons (i.e. limits of dense graphs).  The first step is to normalize these quantities and generalize the
definitions.  The second step is to investigate measurability and continuity of the defined quantities.

Continuity (with respect to the cut metric) allows to transfer properties that have been proved for finite graphs to the
domain of graphons.  However, this goes against the idea behind graphons that some properties of graphs are easier to
understand from an asymptotic viewpoint that is amenable to analytic approaches.  Thus, the goal of this presentation is
to provide direct proofs that do not rely on graph theoretic results.  This works for all results, except for
Lemma~\ref{lem:upper-bound-e}.  Admittedly, the proofs for graphons often imitate the corresponding proofs for graphs.

% Motivated by~\cite{KWPR16:Geometry_of_Degeneracy_Model}.
A further result is the minimal and maximal edge density of a graphon for a given degeneracy.
This generalizes results of \citet{KWPR16:Geometry_of_Degeneracy_Model}.

\section{Preliminaries}
\label{sec:prelim}

\subsection{The case of finite graphs}
\label{sec:finite}

Let $G=(V,E)$ be a (simple, undirected) graph, and let $k\ge 0$ be an integer.
The \emph{$k$-core} of a graph is the maximal subgraph in which every node has degree at least~$k$.  It can be
constructed algorithmically as follows:
\begin{align*}
  K_{k}^{0} &:= V, \\
  K_{k}^{i+1} &:= \big\{ x\in K_{k}^{i} : d_{G|_{K_{k}^{i}}} (x) \ge k \big\}.
\end{align*}
Here, $d_{G|_{K_{k}^{i}}}(x)$ denotes the degree of~$x$ with the induced subgraph~$G|_{K_{k}^{i}}$.  Thus, $K_{k}^{i+1}$
arises from $K_{k}^{i}$ by dropping all nodes that are connected to less than~$k$ other nodes among~$K_{k}^{i}$.  The
sequence $\left(K_{k}^{i}\right)_{i}$ stabilizes after finitely many steps, and the limit $K_{k}:=\bigcap_{i}K_{k}^{i}$
equals the set of nodes of the $k$-core of~$G$.

The \emph{shell index} of a node $x$ of~$G$ is the largest $k$ such that $x\in K_{k}$.  The \emph{degeneracy} of a graph
is the largest shell index of any node.  Alternatively, the degeneracy $\delta(G)$ is the largest $k$ such that $K_{k}$
is non-empty.

% KPPSW14:Statistics_of_core_decompositions,
As shown by \citet[Proposition~3.1]{KWPR16:Geometry_of_Degeneracy_Model}, the degeneracy and the number of edges satisfy
\begin{equation}
  \label{eq:KWPR-inequalities}
  \binom{\delta(G)+1}{2} \le |E| \le \binom{\delta(G)+1}{2} + (n - \delta(G) - 1)\delta(G).
\end{equation}

\subsection{Graphons}
\label{sec:graphons}

For an introduction to graphons (also called graph limits), see \cite{LovaszSzegedy06:Limits_of_dense_graphs} and
\cite{DiaconisJanson08:graph_limits_and_exchangeability}.  In this text, a \emph{graphon} is a (Lebesgue)-measurable
function $w:[0,1]^{2}\to[0,1]$ from the unit square to the unit interval that is symmetric in the sense that
$w(x,y) = w(y,x)$ for all~$x,y\in[0,1]$.

A graphon is a limit object of dense graphs.  In some sense it is analogous to an adjacency matrix.  However, unlike an
adjacency matrix, it can take values other than 0 or 1, and thus its entries do not mark edges, but edge densities.
Similarly, the interval $[0,1]$ that constitutes the domain of each of the two arguments of a graphon $w$ should not be
interpreted as a set of nodes, but rather as a set of node classes.

To relate graphons with finite graphs, consider the following construction: for a finite graph $G$ with $N$ nodes,
consider the random graph on $n\le N$ nodes that arises by randomly sampling $n$ nodes and looking at the induced
subgraph.  Similarly, a subgraph of arbitrary size can be sampled from a graphon~$w$ as follows: Take $n$ i.i.d. samples
$x_{1},\dots,x_{n}$ from the uniform distribution on~$[0,1]$.  Then let $V_{n}=\{1,\dots,n\}$, and for each
$1\le i<j\le n$, add an edge $\{i,j\}$ independently with probability~$w(x_{i},x_{j})$.

Different graphons are considered to be equivalent if the induced probabilities underlying the above sampling process
are identical.
One way to obtain equivalent graphons is by means of measure-preserving maps:
A map $\sigma:[0,1]\to[0,1]$ is \emph{measure-preserving} if it is measurable and satisfies $|A| = |\sigma^{-1}(A)|$ for
any (Lebesgue)-measurable set $A\subseteq[0,1]$.  Given a measure-preserving map $\sigma$, the transformed graphon
$w^{\sigma}$ is defined by
\begin{equation*}
  w^{\sigma}(x,y) = w(\sigma(x), \sigma(y)).
\end{equation*}
% From the perspective of sampling subgraphs from graphons (as presented above), $w$ are $w^{\sigma}$ equivalent.

Measure preserving maps generalize permutations of the nodes of a finite graph.  However, measure preserving maps need
not be injective: For example, the map
\begin{equation*}
  \sigma_{2} : x\mapsto
  \begin{cases}
    2x, & \text{ if }x \le \frac12, \\
    2x - 1, & \text{ if }x > \frac12,
  \end{cases}
\end{equation*}
can be shown to be measure preserving.

Equivalence of graphons is related to the cut metric.  For any pair of graphons $w,w'$, define
\begin{equation*}
  d_{\square}(w, w') = \sup_{S,T\subseteq[0,1]} \left|\int_{S}\mathrm{d}x\int_{T}\emph{d}y \big(w(x,y) - w'(x,y)\big) \right|.
\end{equation*}
The \emph{cut metric} is then defined by
\begin{equation*}
  \delta_{\square}(w, w') = \inf_{\sigma} d_{\square}(w^{\sigma},w'),
\end{equation*}
where the infimum is over all measure-preserving maps~$\sigma$.
The cut metric is a semi-metric on the set of all graphons, and it is a metric on the set of all equivalence classes of
graphons.

\section{The core and the degeneracy of a graphon}

Let $w:[0,1]^{2}\to[0,1]$ be a graphon.  For any $x\in[0,1]$, let $d_{w}(x) = \int_{0}^{1}w(x,y)\mathrm{d}y$ be the
\emph{degree} of $x$ in~$w$.  For any measurable $K\subseteq[0,1]$, let $d^{K}_{w}(x) = \int_{K}w(x,y)\mathrm{d}y$ be
the \emph{degree} of $x$ in~$w$ restricted to~$K$.  The following is a straight-forward generalization of the
corresponding graph theoretic definitions:
\begin{definition}
  The \emph{$\kappa$-core} of~$w$ is the largest subset $K\subseteq[0,1]$
    with $d_{w}^{K}(x)\ge\kappa$ for all~$x\in K$.
  The \emph{shell index} of $x\in[0,1]$ in $w$ is
  \begin{equation*}
    \delta_{x}(w) := \sup\big\{\kappa : x\in K_{\kappa}(w) \big\}.
  \end{equation*}
  The \emph{degeneracy} is
  \begin{equation*}
    \delta(w) := \sup\big\{\kappa : |K_{\kappa}(w)|> 0 \big\},
  \end{equation*}
  where $|K_{\kappa}(w)|$ denotes the (Lebesgue)-volume of~$K_{\kappa}(w)$.
\end{definition}
In fact, as Lemma~\ref{lem:definition-d} below shows,
\begin{equation*}
  \delta(w) = \max\big\{\kappa : |K_{\kappa}(w)|\neq\emptyset\big\} = \max\big\{\delta_{x}(w) : x\in[0,1]\big\}.
\end{equation*}

The algorithmic definition of the $k$-core also generalizes to the graphon case:
for $\kappa\in[0,1]$, let
\begin{align*}
  K_{\kappa}^{1}(w) &:= \big\{ x\in[0,1] : d_{w}(x) \ge\kappa \big\}, \\
  K_{\kappa}^{n+1}(w) &:= \Big\{ x\in K_{\kappa}^{n}(w) : d_{w}^{K_{\kappa}^{n}(w)}(x) \ge\kappa \Big\}.
\end{align*}
Then the $\kappa$-core is $K_{\kappa}(w) := \bigcap_{n=1}^{\infty} K_{\kappa}^{n}(w)$.
\begin{example}
  \label{ex:min-graphon}
  Consider the graphon $w(x,y) = \min\{x,y\}$.  Since $w(x,y)$ is a monotone function in~$x$ for any fixed~$y$, it
  follows that $d_{w}(x)$ is monotone in~$x$, and also $d^{K}_{w}(x)$ increases monotonically with~$x$ for any
  fixed~$K$.  Therefore, each set~$K_{\kappa}^{n}(w)$ is a closed interval of the form $[k_{\kappa}^{n}(w),1]$ (or
  empty).  Thus, each $\kappa$-core is also of the form $K_{\kappa}(w) = [k_{\kappa},1]$ (or empty).  The numbers
  $k_{\kappa}^{n}$ satisfy the recursion
  \begin{equation*}
%    k_{\kappa}^{n+1} = 1 - k_{\kappa}^{n} - \sqrt{1 - 2 k_{\kappa}^{n} + 2(k_{\kappa}^{n})^{2} - 2\kappa},
    k_{\kappa}^{n+1} = 1 - \sqrt{1 - (k_{\kappa}^{n})^{2} - 2\kappa},
    \qquad
    k_{\kappa}^{0} = 0,
  \end{equation*}
  % $\int_{k_{n}}^{1}\min\{x,y\}\drm y = \int_{k_{n}}^{x} y\drm y + \int_{x}^{1} x\drm y
  % = \frac12 x^{2} - \frac12 k_{n}^{2} + x(1-x) \ge \kappa$
  % $x^{2} - 2 x + k_{n}^{2} + 2 \kappa \ge 0$
  % where $K_{\kappa}^{n}(w):=\emptyset$ if $k_{\kappa}^{n}> 1$.
  where the expression inside the squareroot is non-negative if and only if $K_{\kappa}^{n+1}(w)$ is non-empty.
  % For $\kappa\le\frac14$,
  The sequence $(k_{\kappa}^{n})_{n}$ increases (as long as it is defined).  If the sequence does not abort, then it
  approaches a fixed point of the recursion.  The recursion has a fixed point if and only $\kappa\le\frac14$, in which
  case the only stable fixed point lies at
  \begin{equation*}
    k_{\kappa} = \lim_{n\to\infty} k_{\kappa}^{n} = \frac12(1-\sqrt{1 - 4\kappa}).
  \end{equation*}
  Thus, for $\kappa\le\frac14$, the $\kappa$-core is a non-empty interval $K_{\kappa} = [k_{\kappa},1]$.  On the other
  hand, for $\kappa>\frac14$, the sequence $k_{\kappa}^{n}$ has no fixed point, and $K_{\kappa}=\emptyset$.

  Thus, in this example, the cores are a family of intervals $[k_{\kappa},1]$ that are contained in each other and with
  boundaries that depend smoothly on~$\kappa$ for $\kappa\le\frac14$.  The smallest non-empty core is the
  $\frac14$-core $[\frac12,0]$.  Thus, the degeneracy of~$w$ is $\delta(w)=\frac14$.

  The analysis can easily be generalized to other graphons $w(x,y)$ that are monotone in~$x$ for any fixed~$y$.  Such
  graphons arise, for example, in the study of large random graph models associated with a fixed degree
  sequences~\citep{ChatterjeeDiaconisSly11:Random_graphs_with_given_degree_sequence}.  In general, the sets
  $K_{\kappa}^{n}(w)$ and~$K_{\kappa}(w)$ will always be intervals, but the boundaries need not depend continuously
  on~$\kappa$.  However, the next lemma shows that $\bigcap_{\kappa'<\kappa}K_{\kappa'}^{n}(w) = K_{\kappa}^{n}(w)$ and
  $\bigcap_{\kappa'<\kappa}K_{\kappa'}(w) = K_{\kappa}(w)$, which implies that the interval boundaries are upper
  semi-continuous in this case.
\end{example}

\begin{lemma}
  \label{lem:properties-Kn}
  $ $
  \begin{enumerate}
  \item The sets $K^{n}_{\kappa}(w)$ and $K_{\kappa}(w)$ are measurable.
  \item If $\kappa\ge\kappa'$, then $K^{n}_{\kappa}(w) \subseteq K^{n}_{\kappa'}(w)$ and $d_{w}^{K_{\kappa}^{n}}(x)\le
    d_{w}^{K_{\kappa'}^{n}}(x)$ for all~$x$.
  \item Thus, if $\kappa\ge\kappa'$, then $K_{\kappa}(w) \subseteq K_{\kappa'}(w)$.
  \item For each~$n$ and $\kappa>0$, $K_{\kappa}^{n}(w) = \bigcap_{\kappa'<\kappa} K_{\kappa'}^{n}(w)$, and thus
    $K_{\kappa}(w) = \bigcap_{\kappa'<\kappa} K_{\kappa'}^{n}(w)$.
  \item For each~$n$ and $\kappa>0$, $d_{w}^{K_{\kappa}^{n}(w)}(x) = \inf_{\kappa'<\kappa}d_{w}^{K_{\kappa'}^{n}(w)}(x)$
    for all~$x$
  \end{enumerate}
\end{lemma}
\begin{proof}
  (1) follows from induction, since $d_{w}^{K}(x)$ is measurable if $K$ is measurable.  (2) and (3) follow by
  definition.  The last two statements (4) and (5) can be proved simultaneously by induction: for each~$n$, the statement about
  $d_{w}^{K_{\kappa}^{n}(w)}(x)$ % $ = \inf_{\kappa'<\kappa}d_{w}^{K_{\kappa'}^{n}(w)}(x)$ for all~$x$
  follows from the statement about $K_{\kappa}^{n}$ and the monotone convergence theorem for integrals.

  In the case~$n=1$, $K_{\kappa}^{1}(w) = \bigcap_{\kappa'<\kappa} K_{\kappa'}^{1}(w)$ follows from the fact that a
  number $d$ satisfies $d\ge\kappa$ if and only if~$d\ge\kappa'$ for all~$\kappa'<\kappa$.
%
  % The identity $d_{w}^{K_{\kappa}^{1}(w)}(x) = \inf_{\kappa'<\kappa}d_{w}^{K_{\kappa'}^{1}(w)}(x)$ follows for all~$x$
  % from the monotone convergence theorem for integrals.
  For $n>1$, clearly,
  \begin{multline*}
    K_{\kappa}^{n+1}(w) = \Big\{ x\in K_{\kappa}^{n}(w) : d_{w}^{K_{\kappa}^{n}(w)}(x) \ge\kappa \Big\}
    \\
    = \bigcap_{\kappa'<\kappa}\Big\{ x\in K_{\kappa}^{n}(w) : d_{w}^{K_{\kappa}^{n}(w)}(x) \ge\kappa' \Big\}
    \\
    \subseteq \bigcap_{\kappa'<\kappa}\Big\{ x\in K_{\kappa'}^{n}(w) : d_{w}^{K_{\kappa'}^{n}(w)}(x) \ge\kappa' \Big\}
    = \bigcap_{\kappa'<\kappa} K_{\kappa'}^{n+1}(w).
  \end{multline*}
  For the other containment, suppose that $x\notin K_{\kappa}^{n+1}(w)$.
  % By the previous argument, there exists~$\kappa'_{0}<\kappa$ such that
  % \begin{equation*}
  %   x\notin\Big\{ x'\in K_{\kappa'_{0}}^{n}(w) : d_{w}^{K_{\kappa}^{n}(w)}(x') \ge\kappa'_{0} \Big\}.
  % \end{equation*}
  Then either $x\notin K_{\kappa}^{n}(w)$, or $d_{w}^{K^{n}_{\kappa}(w)}(x) < \kappa$.  In the first case, by induction,
  there exists $\kappa'_{0}$ with $x\notin K_{\kappa'_{0}}^{n}(w)$, whence
  $x\notin\bigcap_{\kappa'<\kappa} K_{\kappa'}^{n+1}(w)$.  In the second case, there exists $\kappa'_{0}$ with
  \begin{equation*}
    \kappa > \kappa'_{0} > d_{w}^{K_{\kappa}^{n}(w)}(x) = \inf_{\kappa'<\kappa} d_{w}^{K_{\kappa'}^{n}(w)}(x),
  \end{equation*}
  where the induction hypothesis was used.
  Thus, there exists $\kappa'$ with $\kappa'_{0}<\kappa'<\kappa$ with $d_{w}^{K_{\kappa'}^{n}(w)}(x)\le\kappa'_{0}<\kappa'$.
  Thus, in this case $x\notin K_{\kappa'}^{n+1}(w)$ either.  This finishes
  % the proof of
  % \begin{equation*}
  %   \bigcap_{\kappa'<\kappa}\big\{ x\in K_{\kappa'}^{n}(w) : d_{w}^{K_{\kappa}^{n}(w)}(x) \ge\kappa' \big\}
  %   \subseteq K_{\kappa}^{n+1}(w),
  % \end{equation*}
  % and thus
  the induction step.
\end{proof}
\begin{lemma}
  \label{lem:properties-K}$ $
  \begin{enumerate}
  \item $|K_{\kappa}(w)| = \lim_{n\to\infty}|K_{\kappa}^{n}(w)|$.
  \item If $|K_{\kappa}(w)|>0$, then $|K_{\kappa}(w)|\ge\kappa$.
  \item If $|K_{\kappa}(w)|=0$, then $K_{\kappa}(w)=\emptyset$.  In this case, the sequence $K_{\kappa}^{n}(w)$
    stabilizes after a finite number of steps, i.e.~$K_{\kappa}^{n}(w)=\emptyset$ for some finite~$n$.
  \item $\kappa\mapsto|K_{\kappa}(w)|$ is upper semi-continuous: $|K_{\kappa}(w)| =
    \lim_{\kappa'<\kappa}|K_{\kappa'}(w)| = \limsup_{\kappa'\to\kappa}|K_{\kappa'}(w)|$.
  \end{enumerate}
\end{lemma}
\begin{proof}
  (1) follows from $K_{\kappa}^{n+1}(w)\subseteq K_{\kappa}^{n}(w)$ and the monotone convergence
  theorem. %
  It follows that, if $|K_{\kappa}(w)|<\kappa$, then $|K_{\kappa}^{n}(w)|<\kappa$ for some~$n$.  Then, since
  $w(x,y)\le1$ for all~$x,y\in[0,1]$,
  \begin{equation*}
    d_{w}^{K_{\kappa}^{n}(w)}(x) = \int_{K_{\kappa}^{n}(w)}w(x,y)\mathrm{d}y
    \le |K_{\kappa}^{n}(w)| < \kappa
  \end{equation*}
  for all~$x\in K_{\kappa}^{n}(w)$, whence $K_{\kappa}^{n+1}(w)=\emptyset$.  From this, (2) and (3) follow.

  The last statement follows from Lemma~\ref{lem:properties-Kn}(4).
\end{proof}

As Example~\ref{ex:min-graphon} shows, in general, the sequence $K_{\kappa}^{n}(w)$ for fixed $\kappa$ does not
stabilize after a finite number of steps.

\begin{lemma}
  \label{lem:properties-d} $ $
  \begin{enumerate}
  \item The function $x\mapsto \delta_{x}(w)$ is measurable.
  \item $\left|\big\{ x\in[0,1] : d_{w}(x)\ge\delta(w)\big\}\right|\ge|K_{\delta(w)}|\ge\delta(w)$.
  \item $\inf_{x\in[0,1]} d_{w}(x) \le \delta(w) \le \sup_{x\in[0,1]} d_{w}(x)$.
  \end{enumerate}
\end{lemma}
\begin{proof}
  For each~$\kappa\in[0,1]$,
  $ % \begin{equation*}
    \delta_{w}^{-1}([\kappa,1]) = K_{\kappa}(w)
  $ % \end{equation*}
  is measurable.  This implies (1). % that $\delta_{w}$ is measurable.

  The first inequality in (2) follows from $d_{w}(x)\ge\delta(w)$ for all $x\in K_{\delta(w)}$.  The second inequality
  follows from Lemma~\ref{lem:properties-K}.

  For (3) observe that if $\kappa\le\inf_{x\in[0,1]} d_{w}(x)$, then $[0,1]= K_{\kappa}(w)$.  If $\kappa >
  \sup_{x\in[0,1]} d_{w}(x)$, then $K^{1}_{\kappa}(w)=\emptyset$.
\end{proof}

\begin{lemma}
  \label{lem:definition-d}
  For any~$\kappa>0$ and $x\in[0,1]$,
  \begin{gather*}
    \delta_{x}(w) = \max\big\{\kappa : x\in K_{\kappa}(w)\big\}, \\
    \delta(w) = \max\big\{\kappa : K_{\kappa}(w) \neq\emptyset \big\} = \max\big\{\delta_{x}(w) : x\in[0,1]\big\}.
  \end{gather*}
\end{lemma}
\begin{proof}
  By Lemma~\ref{lem:properties-K}, the conditions $|K_{\kappa}(w)|=0$ and $K_{\kappa}(w)=\emptyset$ are equivalent.  It
  remains to show that the suprema in the definitions of~$\delta_{x}$ and~$\delta$ are actually maxima.  For
  $\delta_{x}$, this follows from statement~4 in Lemma~\ref{lem:properties-Kn}.  For $\delta$, suppose that
  $K_{\kappa}(w) = \emptyset$.  By Lemma~\ref{lem:properties-K}, $K_{\kappa}^{n}(w)=\emptyset$ for some~$n$.  By
  Lemma~\ref{lem:properties-Kn}, there exists $0<\kappa'<\kappa$ with $|K_{\kappa'}^{n}(w)|<\kappa'$, whence
  $K_{\kappa'}(w)=\emptyset$.
\end{proof}

\begin{lemma}
  Let $\sigma$ be a Measure-preserving transformation, and let $w$ be a graphon.
  \begin{enumerate}
  \item $K^{n}_{\kappa}(w^{\sigma}) = \sigma^{-1}(K^{n}_{\kappa}(w))$
  \item $K^{n}_{\kappa}(w) \subseteq \sigma(K^{n}_{\kappa}(w^{\sigma}))$, with
    $|\sigma(K^{n}_{\kappa}(w^{\sigma})) \setminus K^{n}_{\kappa}(w)| = 0$.
  \item $\delta_{x}(w^{\sigma}) = \delta_{\sigma(x)}(w)$ for all~$x\in[0,1]$.
  \item $\delta(w^{\sigma}) = \delta(w)$.
  \end{enumerate}
  % map $\kappa$-cores into $\kappa$-cores (for any $\kappa>0$) and preserve the
  % shelling and the degeneracy.
\end{lemma}
\begin{proof}
  % It suffices to show that $K^{n}_{\kappa}(w^{\sigma}) = \sigma^{-1}(K^{n}_{\kappa}(w))$ for any
  % transformation~$\sigma$ preserving the Lebesgue measure.
  Statements (1), (3) and (4) follow by induction on~$n$, using the following equivalence:
  \begin{align*}
    d^{K}_{w^{\sigma}}(x)\ge\kappa
    &\Longleftrightarrow \int_{K}w(\sigma(x),\sigma(y))\mathrm{d}y \ge\kappa
    \\
    &\Longleftrightarrow \int_{\sigma^{-1}(K)}w(\sigma(x),y)\mathrm{d}y \ge\kappa
%    \\ &
    \Longleftrightarrow d^{\sigma^{-1}(K)}_{w}(\sigma(x))\ge\kappa
  \end{align*}
  that holds for all $\kappa\in[0,1]$ and $K\subseteq[0,1]$.  The first part of (2) follows from~(1), since
  $K^{n}_{\kappa}(w) \subseteq \sigma(\sigma^{-1}(K^{n}_{\kappa}(w))) = \sigma(K^{n}_{\kappa}(w^{\sigma}))$.

  By (3), if $x\in \sigma(K^{n}_{\kappa}(w^{\sigma})) \setminus K^{n}_{\kappa}(w)$, then $x\notin\sigma([0,1])$.  Thus,
  $\sigma(K^{n}_{\kappa}(w^{\sigma})) \setminus K^{n}_{\kappa}(w) \subseteq [0, 1] \setminus \sigma([0,1])$, which is a
  set of measure zero.
\end{proof}

\section{Continuity of the degeneracy}
\label{sec:continuity}

\begin{lemma}
  \label{lem:ineq}
  $|\delta(w) - \delta(w')|\le 2\sqrt{\delta_{\square}(w,w')}$.
\end{lemma}

\begin{proof}
  Recall that $\delta_{\square}(w,w') = \inf_{\sigma} d_{\square}(w,\sigma w')$, where $\sigma$ runs over all invertible
  measure-preserving transformations and where
  \begin{equation*}
    d_{\square}(w,w') = \sup_{S,T\subseteq[0,1]} \Big|\int_{S}\int_{T}\big( w(x,y)-w'(x,y)\big)\mathrm{d}x \mathrm{y}\Big|.
  \end{equation*}
  Since measure-preserving transformations preserve the degeneracy~$\delta$, it suffices to show that
  \begin{equation*}
    |\delta(w) - \delta(w')| \le 2 \sqrt{d_{\square}(w,w')}.
    % \sup_{S,T\subseteq[0,1]} 2 \sqrt{\Big|\int_{S}\int_{T}\big( w(x,y)-w'(x,y)\big)\mathrm{d}x \mathrm{y}\Big|}.
  \end{equation*}
  We may assume that $\epsilon:=\sqrt{d_{\square}(w,w')}>0$.  By symmetry, it is enough to show that $\delta(w') \ge
  \delta(w) - 2 \sqrt{d_{\square}(w,w')}$.  Since $\delta(w')\ge0$, we may assume that % $\delta(w)>0$ and that
  $\epsilon=\sqrt{d_{\square}(w,w')}<\frac12\delta(w)<\delta(w)$.

  Let $K$ be the $\delta(w)$-core of~$w$, and let
  \begin{equation*}
    S_{\epsilon} :=
    \Big\{
    x\in K: d^{K}_{w'}(x) < \delta(w) - \epsilon
    \Big\}.
  \end{equation*}
  Then
  \begin{equation*}
    d_{\square}(w,w') \ge \int_{S_{\epsilon}}\mathrm{d}x\int_{K}\mathrm{d}y\big(w(x,y) - w'(x,y)\big)
    > |S_{\epsilon}|\delta(w) - |S_{\epsilon}|(\delta(w)-\epsilon) = \epsilon |S_{\epsilon}|,
  \end{equation*}
  and so
  \begin{equation*}
    |S_{\epsilon}| < \frac{d_{\square}(w,w')}{\epsilon} = \sqrt{d_{\square}(w,w')} < \delta(w) \le |K|.
  \end{equation*}

  Let $K':= K\setminus S_{\epsilon}$.  Then $|K'|>0$.  Moreover, for any $x\in K'$,
  \begin{multline*}
    \int_{K'} w'(x,y)\mathrm{d}y = \int_{K} w'(x,y)\mathrm{d}y - \int_{S_{\epsilon}} w'(x,y)\mathrm{d}y
    \\
    \ge d^{K}_{w'}(x) - |S_{\epsilon}| > (\delta(w) - \epsilon) - \frac{d_{\square}(w,w')}{\epsilon} = \delta(w) -
    2\sqrt{d_{\square}(w,w')}.
  \end{multline*}
  This implies that $K'$ is part of the $(\delta(w)-2\sqrt{d_{\square}(w,w')})$-core.  Thus,
  \begin{equation*}
    \delta(w') \ge \delta(w) - 2\sqrt{d_{\square}(w,w')}. \qedhere
  \end{equation*}
\end{proof}

Lemma~\ref{lem:ineq} shows that the degeneracy is a continuous graph invariant.  The square root suggests that the
degeneracy is Hölder continuous, but not Lipschitz continuous.  The following example illustrates this:
\begin{example}
  Let $a,b,\alpha\in(0,1)$.  %
  Let $w(x,y) = a$ be the constant graphon, and let
  \begin{equation*}
    w'(x,y) =
    \begin{cases}
      b, & \text{ if }x\le\alpha,y\le\alpha, \\
      a, & \text{ otherwise.}
    \end{cases}
  \end{equation*}
  Since $w$ is invariant under all measure-preserving transformations,
  \begin{equation*}
    \delta_{\square}(w,w')
    = d_{\square}(w,w')
    = \alpha^{2}|a-b|.
  \end{equation*}
  It suffices to compute
  \begin{equation}
    \label{eq:delta-wp}
    \delta(w') =
    \begin{cases}
      (1-\alpha)a+\alpha b, & \text{ if }b<a, \\
      \max\{a, \alpha b\}, & \text{ if }b\ge a.
    \end{cases}
  \end{equation}
  This shows that $\delta(w')-\delta(w)$ varies linearly with~$\alpha$ for $b<a$, while $\delta_{\square}(w,w')$ is
  quadratic in~$\alpha$.

  It remains to prove~\eqref{eq:delta-wp}.  Since
  \begin{equation*}
    d_{w'}(x) =
    \begin{cases}
      (1-\alpha)a+\alpha b, & \text{ if }x\le\alpha, \\
      a, & \text{ if }x>\alpha,
    \end{cases}
  \end{equation*}
  it follows that $K_{\min\{(1-\alpha)a+\alpha b,a\}}=[0,1]$, whence $\delta(w')\ge\min\{(1-\alpha)a+\alpha b,a\}$.
  Assume that $b<a$, and let $\kappa>(1-\alpha)a+\alpha b$.  Then $K^{1}_{\kappa}=[\alpha,1]$ and
  $d^{K^{1}_\kappa(w')}(x)\le(1-\alpha)a<\kappa$ for all~$x$.  Hence,
  $K_{\kappa}^{2}(w')=\emptyset=K_{\kappa}(w')$, which proves that $\delta(w')=(1-\alpha)a+\alpha b$ in this case.

  If $b\ge a$, then $\delta(w')\ge a$.  Since $[0,\alpha]\subseteq K_{\alpha b}(w')$, it follows that
  $\delta(w')\ge\alpha b$.  If $\kappa>a$, then $K^{1}_{\kappa}(w')\subseteq[0,\alpha]$, whence
  $d^{K^{1}_{\kappa}(w')}_{w'}(x)\le\alpha b$.  Thus, if $\kappa>\max\{\alpha b,a\}$, then
  $K^{2}_{\kappa}(w')=\emptyset$, whence $\delta(w')\le\max\{a,\alpha b\}$.
\end{example}

\section{Degeneracy and edge density}

For any graphon~$w$, let
\begin{equation*}
  e(w) = \iint_{[0,1]^{2}}\mathrm{d}x\,\mathrm{d}y\, w(x,y)
\end{equation*}
be the \emph{edge density} of~$w$.
\begin{lemma}
  \label{lem:lower-bound-e}
  $e(w)\ge\delta(w)^{2}$, with equality if and only if $\delta_{\square}(w,w')=0$, where
  \begin{equation*}
    w'(x,y)=
    \begin{cases}
      1, & \text{ if }x,y\le\delta(w), \\
      0, & \text{ otherwise.}
    \end{cases}
  \end{equation*}
\end{lemma}
This lemma follows from the corresponding result about graphs (see
\citet[Proposition~3.1]{KWPR16:Geometry_of_Degeneracy_Model}).  It is also insightfull to prove it directly, mimicking
the proof of the result for graphs.
\begin{proof}
  Let $w$ be a graphon with $\delta(w)$-core~$K$.  The restricted graphon
  \begin{equation*}
    w_{K}(x,y) :=
    \begin{cases}
      w(x,y), & \text{ if }x,y\in K, \\
      0, & \text{ otherwise,}
    \end{cases}
  \end{equation*}
  has the same degeneracy as~$w$, the same $\delta(w)$-core and satisfies $e(w_{K})\le e(w)$, with strict inequality if
  the set of pairs $x,y\in[0,1]\setminus K$ with $w(x,y)>0$ has measure larger than zero.  Hence, without loss of
  generality, we may assume that~$w=w_{K}$; that is, $w(x,y)=0$ if $x\notin K$ or $y\notin K$.

  By Lemma~\ref{lem:properties-d},
  \begin{equation*}
    e(w) = \int_{K}\mathrm{d}x\int_{0}^{1}\mathrm{d}y\, w(x,y)
    = \int_{K}\mathrm{d}x \, d_{w}(x) \ge |K|\delta(w) \ge \delta(w)^{2}.
  \end{equation*}
  Equality holds if and only if (i) $|K|=\delta(w)$ and (ii) $d_{w}(x)=\delta(w)$ for almost all~$x\in K$.  Condition
  (i) implies $w(x,y)=1$ for all~$x,y\in K$.  There is a measure-preserving transformation $\sigma:[0,1]\to[0,1]$ with
  $\sigma(K) = [0,\delta(w)]$ and $K = \sigma^{-1}([0,\delta(w)])$.
  Then $w = (w')^{\sigma}$,  % $d_{\square}(w, (w')^{\sigma}) = 0$,
  with $w'$ as in the statement of the lemma.
  % Applying a measure-preserving transformation, we may assume that $K=[0,\delta(w)]$.
\end{proof}

\begin{lemma}
  \label{lem:upper-bound-e}
  $e(w)\le\delta(w)(2 - \delta(w))$.
  For any~$\delta\in(0,1)$, equality holds for the graphon
  \begin{equation*}
    w_{\delta}(x,y) =
    \begin{cases}
      1, & \text{ if } \max\{x,y\} \ge 1 - \delta, \\
      0, & \text{ otherwise.}
    \end{cases}
  \end{equation*}
  (with $\delta(w_{\delta}) = \delta$).
\end{lemma}
Lemma~\ref{lem:upper-bound-e} follows from Proposition~3.1 by \citet{KWPR16:Geometry_of_Degeneracy_Model}.  Unlike for
the other results, no alternative proof will be presented, as it seems to be difficult to directly prove this result in
the graphon world.  The proof of \citeauthor{KWPR16:Geometry_of_Degeneracy_Model} relies on Proposition~3.10 by
\citet{KPPSW14:Statistics_of_core_decompositions}, which starts by ordering the nodes of a graph according to their
shell index.  Among nodes with the same shell index~$k$, nodes are ordered according to how quickly they disappear in
the sequence $(K_{k+1}^{i})_{i}$.  Generalizing this approach to graphons poses two problems.  First, in the continuous
case, a shell index $\kappa\in[0,1]$ does not have a ``successor'' $\kappa+\epsilon$, as the set of shell indices may be
infinite.  The appendix contains an example that illustrates the difficulty of ordering the nodes.
% Maybe this problem could be overcome by choosing $\epsilon>0$ small enough, maybe dependent on $\kappa$ and by
% considering a suitable limit.  However, it remains the second problem that
Second, instead of constructing a permutation that orders the nodes one needs to construct an invertible measure
preserving map.

\subsubsection*{Acknowledgments}

Thanks goes to Dane Wilburne, whose questions motivated the results of this paper.

\bibliographystyle{plainnat}
\bibliography{general}

\begin{thebibliography}{5}
\providecommand{\natexlab}[1]{#1}
\providecommand{\url}[1]{\texttt{#1}}
\expandafter\ifx\csname urlstyle\endcsname\relax
  \providecommand{\doi}[1]{doi: #1}\else
  \providecommand{\doi}{doi: \begingroup \urlstyle{rm}\Url}\fi

\bibitem[Chatterjee et~al.(2011)Chatterjee, Diaconis, and
  Sly]{ChatterjeeDiaconisSly11:Random_graphs_with_given_degree_sequence}
Sourav Chatterjee, Persi Diaconis, and Allan Sly.
\newblock Random graphs with a given degree sequence.
\newblock \emph{Annals of Applied Probability}, 21\penalty0 (4):\penalty0
  1400--1435, 2011.

\bibitem[Diaconis and
  Janson(2008)]{DiaconisJanson08:graph_limits_and_exchangeability}
Persi Diaconis and Svante Janson.
\newblock Graph limits and exchangeable random graphs.
\newblock \emph{Rendiconti di Matematica}, 28:\penalty0 33--61, 2008.

\bibitem[Karwa et~al.(2014)Karwa, Pelsmajer, Petrović, Stasi, and
  Wilburne]{KPPSW14:Statistics_of_core_decompositions}
Vishesh Karwa, Michael Pelsmajer, Sonja Petrović, Despina Stasi, and Dane
  Wilburne.
\newblock Statistical models for cores decomposition of an undirected random
  graph.
\newblock \emph{arXiv:1410.7357}, 2014.

\bibitem[Kim et~al.(2016)Kim, Wilburne, Petrović, and
  Rinaldo]{KWPR16:Geometry_of_Degeneracy_Model}
Nicolas Kim, Dane Wilburne, Sonja Petrović, and Alessandro Rinaldo.
\newblock On the geometry and extremal properties of the edge-degeneracy model.
\newblock \emph{arXiv:1602.00180}, 2016.

\bibitem[Lovász and Szegedy(2006)]{LovaszSzegedy06:Limits_of_dense_graphs}
Lászlo Lovász and Balázs Szegedy.
\newblock Limits of dense graph sequences.
\newblock \emph{Journal of Combinatorial Theory, Series B}, 96\penalty0
  (6):\penalty0 933--957, 2006.

\end{thebibliography}

\newpage{}
\appendix
\allowdisplaybreaks{}

\section{Another example}
\label{sec:another-example}

Let $\epsilon=(\epsilon_{i})_{i\in\Nb}, \epsilon'=(\epsilon'_{i})_{i\in\Nb}$ be two monotonically decreasing sequences
of positive real numbers with
\begin{equation*}
  % \epsilon_{1} = \epsilon_{2}
  % \quad\text{ and }\quad
  % \epsilon'_{1} = \epsilon'_{2}
  % \quad\text{ and }\quad
  \sum_{i\in\Nb}\epsilon_{i} = 1 = \sum_{i\in\Nb}\epsilon'_{i}.
\end{equation*}
Let $\alpha_{i}=\sum_{i'=1}^{i}\epsilon_{i'}$ and $\alpha'_{i}=\sum_{i'=1}^{i}\epsilon'_{i'}$.  Define a graphon
$w_{\epsilon,\epsilon'}(x, y)$ as follows: for all $x,y\in[0,1]$ with $x<y$, let
\begin{equation*}
  w_{\epsilon,\epsilon'}(x, y) =
  \begin{cases}
    1, & \text{ if } \frac15(1 - \alpha_{i+1}) \le x < \frac15(1 - \alpha_{i}) \\ &\qquad \text{ and } \frac15(1 - \alpha_{i}) \le y < \frac15(1 - \alpha_{i-1}), \\
    1 - \epsilon_{i-1}, & \text{ if } \frac15(1 - \alpha_{i}) \le x < \frac15(1 - \alpha_{i-1}) \\ &\qquad \text{ and } \frac15 \le y < \frac25, \\
    1, & \text{ if }  \frac15 \le x < \frac25 \text{ and } \frac25 \le y < \frac35, \\
    1, & \text{ if }  \frac25 \le x < \frac35 \text{ and } \frac35 \le y < \frac45, \\
    1 - \epsilon'_{i-1}, & \text{ if } \frac35 \le x < \frac45   \\ &\qquad \text{ and } \frac15(4 + \alpha'_{i-1}) \le y < \frac15(4 + \alpha'_{i}), \\
    1, & \text{ if } \frac15(4 + \alpha'_{i-1}) \le x < \frac15(4 + \alpha'_{i})  \\ &\qquad \text{ and } \frac15(4 + \alpha'_{i}) \le y < \frac15(4 + \alpha'_{i+1}), \\
    0, & \text{ otherwise,}
  \end{cases}
\end{equation*}
where $\epsilon_{0}=\epsilon'_{0}=0=\alpha_{0}=\alpha'_{0}$.
% The assumptions on $\epsilon$ ensure that $0\le 1 + \epsilon_{i} - \epsilon_{i-1} - \epsilon_{i+1}\le 1$, and similarly for~$\epsilon'$.
In a picture:
\begin{center}
  \begin{tabular}{|c|c|c|c|c|}
    \hline
    \raisebox{3mm}{
      \begin{tabular}{ccccc}
        \scriptsize$\ddots$ \\
        & \tiny        0 & \scriptsize  1 & \footnotesize0 & \small 0 \\
        & \scriptsize  1 & \footnotesize0 & \small1 &        0 \\
        & \footnotesize0 & \small       1 &       0 & \large 1 \\
        & \small       0 &              0 & \large1 & \large 0
      \end{tabular}
    }
    & \raisebox{3mm}{
      \begin{tabular}{c}
        \scriptsize $\vdots$ \\
        \small$1 - \epsilon_{3}$ \\
        $1 - \epsilon_{2}$ \\
        \large$1 - \epsilon_{1}$ \\
        \large$1$
      \end{tabular}
} &  \mbox{\Huge{}0} & \mbox{\Huge{}0} &  \mbox{\Huge{}0} \\[-2mm] &&&&\\\hline&&&&\\[-3mm]
    \raisebox{1mm}{$1 - \epsilon_{i-1}$} &  \mbox{\Huge{}\phantom{$^{0}$}0\phantom{$^{0}$}} &  \mbox{\Huge{}1} & \mbox{\Huge{}0} &  \mbox{\Huge{}0} \\[1mm]\hline&&&&\\[-3mm]
    \mbox{\Huge{}0} & \mbox{\Huge{}1} &  \mbox{\Huge{}\phantom{$^{0}$}0\phantom{$^{0}$}} &  \mbox{\Huge{}1} & \mbox{\Huge{}0} \\[1mm]\hline&&&&\\[-3mm]
    \mbox{\Huge{}0} & \mbox{\Huge{}0} & \mbox{\Huge{}1} &  \mbox{\Huge{}\phantom{$^{0}$}0\phantom{$^{0}$}} &  \raisebox{1mm}{$1 - \epsilon'_{i-1}$} \\[1mm]\hline&&&&\\[-1mm]
    \mbox{\Huge{}0} & \mbox{\Huge{}0} &  \mbox{\Huge{}0} &
    \mbox{
      \begin{tabular}{c}
        \large$1$ \\
        \large$1 - \epsilon'_{1}$ \\
        $1 - \epsilon'_{2}$ \\
        \small$1 - \epsilon'_{3}$ \\
        \scriptsize $\vdots$
      \end{tabular}
} &
    \mbox{
      \begin{tabular}{ccccc}
        \large 0 & \large 1 & \large 0 & \large 0\rlap{\phantom{$^{0}$}} \\
        \large 1 & 0 & 1 & 0 \\
        \large 0 & 1 & \small 0 & \small 1 \\
        \large 0 & 0 & \small 1 & \scriptsize 0 \\
        & & & & \scriptsize$\ddots$
      \end{tabular}
    }
    \\ \hline
  \end{tabular}
\end{center}

The calculation in the following will show that each $K^{n}_{\kappa}(w)$ is of the form
$\big[l_{n}(\kappa); u_{n}(\kappa)\big)$ or empty, with strictly monotonically increasing sequences $l_{n}(\kappa)$,
$u_{n}(\kappa)$.  Claim: \emph{It is possible to construct the sequences $\epsilon,\epsilon'$ such that for any
  $\kappa_{0}>\frac15$ there are $\kappa,\kappa'>\frac15$ with $\kappa,\kappa'<\kappa_{0}$ and
\begin{equation*}
  l_{n}(\kappa) \le \frac25 < u_{n}(\kappa) < \frac35,
  \quad\text{ and }\quad
  \frac25 < l_{n'}(\kappa) \le \frac35 < u_{n'}(\kappa)
\end{equation*}
for some $n,n'$.}  This shows that it is difficult to generalize the idea of the proof Proposition~3.10 by
\citet{KPPSW14:Statistics_of_core_decompositions} of ordering the nodes of a graph by how fast they disappear in the
sequence $K^{n}_{\kappa}$. % to graphons.

One computes
\begin{equation*}
  d_{w}(x) =
  \begin{cases}
    \frac15(1 + \epsilon_{i+1}), & \text{ if } \frac15(1 - \alpha_{i}) \le x < \frac15(1 - \alpha_{i-1}), \; i\ge 1,\\
    \frac15\big(1 + \beta_{\infty}\big), & \text{ if }  \frac15 \le x < \frac25, \\
    \frac25, & \text{ if }  \frac25 \le x < \frac35, \\
    \frac15\big(1 + \beta'_{\infty}\big), & \text{ if }  \frac35 \le x < \frac45, \\
    \frac15(1 + \epsilon'_{i+1}), & \text{ if } \frac15(4 + \alpha'_{i-1}) \le x < \frac15(4 + \alpha'_{i}), \; i\ge 1,
  \end{cases}
\end{equation*}
where $\beta_{j} = \sum_{i=1}^{j}\epsilon_{i}(1 - \epsilon_{i-1})$ and
$\beta'_{j} = \sum_{i=1}^{j}\epsilon'_{i}(1 - \epsilon'_{i-1})$.  The assumptions on $\epsilon,\epsilon'$ imply
$0\le\beta_{j},\beta'_{j}\le 1$ and $\beta_{j}\le\beta_{j+1}$ and $\beta'_{j}\le\beta'_{j+1}$ for $j=1,2,\dots,\infty$.
It follows that $K_{\kappa}(w) = [0,1]$ for $\kappa\le\frac15$.

Let $\kappa>\frac15$ be such that $\kappa\le\frac15\big(1 + \min\{\beta_{1},\beta'_{1}\}\big)$.  Denote by
$i_{\kappa},i'_{\kappa}$ the smallest positive integers with $\frac15(1 + \epsilon_{i_{\kappa}+1}) < \kappa$ and
$\frac15(1 + \epsilon'_{i'_{\kappa}+1}) < \kappa$.  Then
$K_{\kappa}^{1}(w) = \big[\frac15(1 - \alpha_{i_{\kappa}-1}), \frac15(4 + \alpha'_{i'_{\kappa}-1})\big)$, and
\begin{equation*}
  d^{K_{\kappa}^{1}(w)}_{w}(x) =
  \begin{cases}
%    0, & \text{ if } x < \frac15(1 - \alpha_{i_{\kappa}}), \\
    \frac15, & \text{ if } \frac15(1 - \alpha_{i_{\kappa} - 1}) \le x < \frac15(1 - \alpha_{i_{\kappa} - 2}), \\
    \frac15(1 + \epsilon_{i+1}), & \text{ if } \frac15(1 - \alpha_{i}) \le x < \frac15(1 - \alpha_{i-1}),\; 1 \le i < i_{\kappa} - 1\\
    \frac15\big(1 + \beta_{i_{\kappa} - 1}\big), & \text{ if }  \frac15 \le x < \frac25, \\
    \frac25, & \text{ if }  \frac25 \le x < \frac35, \\
    \frac15\big(1 + \beta'_{i'_{\kappa} - 1}\big), & \text{ if }  \frac35 \le x < \frac45, \\
    \frac15(1 + \epsilon'_{i+1}), & \text{ if } \frac15(4 + \alpha'_{i-1}) \le x < \frac15(4 + \alpha'_{i}),\; 1 \le i < i'_{\kappa} - 1, \\
    \frac15, & \text{ if } \frac15(4 + \alpha_{i'_{\kappa} - 2}) \le x < \frac15(4 + \alpha_{i'_{\kappa} - 1}), % \\
%    0, & \text{ if } x \ge \frac15(4 + \alpha'_{i'_{\kappa}}),
  \end{cases}
\end{equation*}
for $x\in K_{\kappa}^{1}(w)$.

Suppose that $i_{\kappa}>1$.  If $\frac15(1 - \alpha_{i_{\kappa} - 1}) \le x < \frac15(1 - \alpha_{i_{\kappa} - 2})$,
then $x\notin K^{2}_{\kappa}(w)$.  On the other hand, the interval $\big[\frac15(1 - \alpha_{i_{\kappa}-2}); \frac12\big]$
belongs to~$K^{2}_{\kappa}(w)$.  Similarly, if $i'_{\kappa}>1$, then $K^{2}_{\kappa}(w)$ contains
$\big[\frac12; \frac15(4 + \alpha'_{i'_{\kappa}-2})\big)$, but does not contain any $x$ with
$\frac15(4 + \alpha_{i_{\kappa} - 1}) \le x < \frac15(4 + \alpha_{i_{\kappa}})$.  Using induction, it follows that
\begin{equation*}
  K^{k}_{\kappa}(w) = \big[\frac15(1 - \alpha_{i_{\kappa} - k}), \frac15(4 + \alpha'_{i'_{\kappa} - k})\big]
\end{equation*}
for $k=1,\dots,\min\{i_{\kappa},i'_{\kappa}\}$.

Assume that $i_{\kappa}\le i'_{\kappa}$.  Then
$K^{i_{\kappa}}_{\kappa}(w) = \big[\frac15, \frac15(4 + \alpha'_{i'_{\kappa} - i_{\kappa}})\big)$,
and
\begin{equation*}
  d^{K_{\kappa}^{i_{\kappa}}}_{w}(x) =
  \begin{cases}
    \frac15, & \text{ if }  \frac15 \le x < \frac25, \\
    \frac25, & \text{ if }  \frac25 \le x < \frac35, \\
    \frac15\big(1 + \beta'_{i'_{\kappa}-i_{\kappa}}\big), & \text{ if }  \frac35 \le x < \frac45, \\
    \frac15(1 + \epsilon'_{i + 1}), & \text{ if } \frac15(4 + \alpha'_{i-1}) \le x < \frac15(4 + \alpha'_{i}), \; 1 \le i < i'_{\kappa} - i_{\kappa}, \\
    \frac15, & \text{ if } \frac15(4 + \alpha_{i'_{\kappa} - i_{\kappa} - 1}) \le x < \frac15(4 + \alpha_{i'_{\kappa} - i_{\kappa}}),
  \end{cases}
\end{equation*}
for $x\in K^{\kappa^{i_{\kappa}}}(w)$.  If $i'_{\kappa} = i_{\kappa}$, then
$K^{i_{\kappa} + 1}_{\kappa}(w) = [\frac25;\frac35)$.  If $i'_{\kappa} = i_{\kappa} + 1$, then
$K^{i_{\kappa} + 1}_{\kappa}(w) = [\frac25;\frac45)$.  Thus, if $i'_{\kappa}\in\{i_{\kappa}, i_{\kappa} + 1\}$, then
$K^{i_{\kappa} + 2}_{\kappa}(w) = \emptyset$.

If $i'_{\kappa} \ge i_{\kappa} + 2$, then
$K^{i_{\kappa} + 1}_{\kappa}(w) = \big[\frac25;\frac15(4 + \alpha_{i'_{\kappa} - i_{\kappa} - 1})\big)$, and thus
\begin{equation*}
  d^{K_{\kappa}^{i_{\kappa} + 1}}_{w}(x) =
  \begin{cases}
    \frac15, & \text{ if }  \frac25 \le x < \frac35, \\
    \frac15\big(1 + \beta'_{i'_{\kappa}-i_{\kappa}-1}\big), & \text{ if }  \frac35 \le x < \frac45, \\
    \frac15(1 + \epsilon'_{i + 1}), & \text{ if } \frac15(4 + \alpha'_{i-1}) \le x < \frac15(4 + \alpha'_{i}), \\
                                     &  \qquad\qquad\qquad\qquad\; 1 \le i < i'_{\kappa} - i_{\kappa}-1, \\
    \frac15, & \text{ if } \frac15(4 + \alpha_{i'_{\kappa} - i_{\kappa} - 2}) \le x < \frac15(4 + \alpha_{i'_{\kappa} - i_{\kappa} - 1}),
  \end{cases}
\end{equation*}
for $x\in K^{i_{\kappa} + 1}_{\kappa}(w)$.  Therefore,
$K_{\kappa}^{i_{\kappa} + 2} = \big[\frac35;\frac15(4 + \alpha_{i'_{\kappa} - i_{\kappa} - 2})\big)$.
If $i'_{\kappa} = i_{\kappa} + 2$, then $K_{\kappa}^{i_{\kappa} + 3} = \emptyset$.

Otherwise, if $i'_{\kappa} > i_{\kappa} + 2$, then
\begin{equation*}
  d^{K_{\kappa}^{i_{\kappa} + 2}}_{w}(x) =
  \begin{cases}
    \frac15\beta'_{i'_{\kappa}-i_{\kappa}-2}, & \text{ if }  \frac35 \le x < \frac45, \\
    \frac15(1 + \epsilon'_{i + 1}), & \text{ if } \frac15(4 + \alpha'_{i-1}) \le x < \frac15(4 + \alpha'_{i}), \\
                                     &  \qquad\qquad\qquad\qquad\; 1 \le i < i'_{\kappa} - i_{\kappa}-2, \\
    \frac15, & \text{ if } \frac15(4 + \alpha_{i'_{\kappa} - i_{\kappa} - 3}) \le x < \frac15(4 + \alpha_{i'_{\kappa} - i_{\kappa} - 2}),
  \end{cases}
\end{equation*}
whence $K_{\kappa}^{i_{\kappa} + 3} = \big[\frac45;\frac15(4 + \alpha_{i'_{\kappa}- i_{\kappa} - 3})\big)$, and
$K_{\kappa}^{i_{\kappa} + 4} = \emptyset$.

In any case, if $i_{\kappa} < i'_{\kappa}$, then
$[\frac25;\frac35)\subseteq K^{i_{\kappa}+1}_{\kappa}\setminus K^{i_{\kappa}+2}_{\kappa}$ and
$[\frac35;\frac45)\subseteq K^{i_{\kappa}+2}_{\kappa}$ (that is, any $x\in[\frac25;\frac35)$ disappears before any
$y\in[\frac3;\frac45)$.  Conversely, if $i_{\kappa} < i'_{\kappa}$, then
$[\frac35;\frac45)\subseteq K^{i'_{\kappa}}_{\kappa}\setminus K^{i'_{\kappa}+1}_{\kappa}$ and
$[\frac25;\frac35)\subseteq K^{i'_{\kappa}+1}_{\kappa}$.

To prove the claim, it suffices to construct the sequences $\epsilon,\epsilon'$ such that for any $\kappa_{0}>\frac15$
there are $\kappa,\kappa'>\frac15$ with $\kappa,\kappa'<\kappa_{0}$ and $i_{\kappa} < i'_{\kappa}$ and
$i_{\kappa'} > i'_{\kappa'}$.
Consider the function $f(n) = 1 - \frac{1}{n + 1}$, and let
\begin{align*}
  \alpha_{n} &=
  \begin{cases}
    \frac{f(n-1) + f(n+1)}{2}, & \text{ if $n$ is odd},\\
    f(n), & \text{ if $n$ is even},
  \end{cases}
                                 &
  \alpha'_{n} &=
  \begin{cases}
    f(n), & \text{ if $n$ is odd}, \\
    \frac{f(n-1) + f(n+1)}{2}, & \text{ if $n$ is even}.
  \end{cases}
\end{align*}
As $f$ is concave, the sequences $\epsilon_{n} = \alpha_{n}- \alpha_{n-1}$, $\epsilon'_{n} = \alpha'_{n}- \alpha'_{n-1}$
are decreasing.  Moreover,
\begin{equation*}
  \epsilon'_{1} > \epsilon_{1} = \epsilon_{2} > \epsilon'_{2} = \epsilon'_{3} > \epsilon_{4} = \dots
\end{equation*}
Let $\kappa_{i} = \frac15\big( + \frac{\epsilon_{i} + \epsilon'_{i}}{2}\big)$.  Then
$\frac15(1 + \epsilon'_{i})>\kappa_{i}>\frac15(1 + \epsilon_{i})$ if $i$ is odd, and
$\frac15(1 + \epsilon_{i})>\kappa_{i}>\frac15(1 + \epsilon'_{i})$ if $i$ is even.  Thus,
\begin{equation*}
  i_{\kappa_{i}} =
  \begin{cases}
    i - 1, & \text{ if $i$ is odd}, \\
    i, & \text{ if $i$ is even},
  \end{cases}
  \quad\text{ and }\quad
  i'_{\kappa_{i}} =
  \begin{cases}
    i, & \text{ if $i$ is odd}, \\
    i - 1, & \text{ if $i$ is even}.
  \end{cases}
\end{equation*}
Thus, the sequences $\epsilon,\epsilon'$ satisfy the claim.

\end{document}